\documentclass[a4paper,oneside,english,reqno,12pt]{amsart}
\usepackage[utf8]{inputenc}
\usepackage[T1]{fontenc}
\usepackage{graphicx}
\DeclareFontFamily{OMX}{mlmex}{}
\DeclareFontShape{OMX}{mlmex}{m}{n}{%
   <->mlmex10%
   }{}%
\usepackage{mlmodern}

\usepackage[hscale=0.666, vscale=0.72]{geometry}
\usepackage{babel}
\usepackage[numbers,sort&compress]{natbib}

\usepackage[np,autolanguage]{numprint}
\usepackage{hyperref}
\hypersetup{%
  colorlinks=true,%
  citecolor=[RGB]{120,29,126},%
  pdfauthor={Jean-Fran\c{c}ois Burnol},%
  pdfsubject={Zeta function, analytic combinatorics},%
  pdfstartview=FitH,%
  pdfpagemode=UseNone,%
}

\usepackage{amsmath,amsthm,amssymb}
\usepackage{mathtools}

\DeclarePairedDelimiterX\Iffint[2]{\lbrack\!\lbrack}{\rbrack\!\rbrack}{#1\dots#2}
\DeclarePairedDelimiterX\Ioo[2]{\lparen}{\rparen}{#1,#2}
\DeclarePairedDelimiterX\Iof[2]{\lparen}{\rbrack}{#1,#2}
\DeclarePairedDelimiterX\Ifo[2]{\lbrack}{\rparen}{#1,#2}
\DeclarePairedDelimiterX\Iff[2]{\lbrack}{\rbrack}{#1,#2}

\newcommand\ZZ{\mathbb{Z}}
\newcommand\RR{\mathbb{R}}
\newcommand\PP{\mathbb{P}}
\newcommand\CC{\mathbb{C}}
\newcommand\EE{\mathbb{E}}

\newcommand\e{\mathsf{e}}
\newcommand\dt{\mathrm{d}t}
\newcommand\du{\mathrm{d}u}

\newcommand\dz{\mathrm{d}z}

\newcommand\dmu{\mathrm{d}\mu}

\DeclareMathOperator\Res{Res}
\DeclareMathOperator\Arg{Arg}

\theoremstyle{plain}
\newtheorem{theo}{Theorem}
\newtheorem{prop}{Proposition}

\theoremstyle{definition}

\newtheorem{rema}{Remark}

\allowdisplaybreaks

\title[Log-periodicity in asymptotics]
{Full asymptotics and log-periodicity for coefficients in representations of
  the Riemann zeta function}

\author[J.-F. Burnol]
 {Jean-Fran\c{c}ois Burnol}

 \date{29 September 2026 for v6, which is essentially the same as v5 apart
   from having a new title and abstract. The last section with a plot of
   $u_m^*(s)$ for $b=2$ and $s=\frac12$ was dropped. See \cite{burnolosc} for
   many such examples, in the more general case of Dirichlet series with
   missing digits, but only with $s$ in in the half-plane of convergence
   though.}


\usepackage{setspace}

\begin{document}

\begin{abstract}
  For each integer $b\ge2$, we obtain representations of the Riemann zeta
  function in the complex plane as geometrically convergent series. Mellin
  transform methods and asymptotics for ratios of Gamma functions from the
  author's recent work allow us to prove, for the coefficients of these
  series, asymptotic expansions to all orders in inverse powers of the index,
  with numerators which are periodic functions of the base $b$ logarithm of
  the index. Such complete asymptotics were first established in the author's
  recent work for quantities which arose decades ago in theoretical computer
  science.
\end{abstract}

\maketitle

\onehalfspacing

\section{Introduction}

Let $b>1$ be an integer.  The paper is devoted to the following sequence,
indexed by non-negative integers, of meromorphic
functions of the complex variable $s$:
\begin{equation}
  \label{eq:umber}
   u_m(s) =  \frac1{m+1}\sum_{j=0}^{m}\binom{m+1}{j}B_j\frac{b^{s+j}}{b^{s+j}-b}\;.
\end{equation}
They are rational functions of $b^{-s}$. The $B_j$'s are the Bernoulli
numbers, with $B_1 = -\frac12$. For $m\geq1$, the numerator $b^{s+j}$ can be
replaced by $b$ without changing the value of the finite sum.

Here is how the contents are organized.  First, restricting to $b=2$, we
explain that $u_m(s)$ can be expressed in terms of a probability generating
function which arises in the analysis of the \emph{leader election algorithm}:
i.e.\@ the selection of a winner, out of a group of $m+1$ people, via repeated
coin tosses.  See \cite{prodinger1993}, \cite{fillmahmszp1996},
\cite[10.5.1]{szpbook2001} about this.  See
further \cite{janszp1997} for the case with biased coins.  A quantity of
interest is the expected duration of the process. It turns out to be, for
$m\geq1$:
\begin{equation}\label{eq:em}
  e_m = 1 - \sum_{k=1}^m \binom{m+1}{k}B_k\frac{1}{2^k-1}\;.
\end{equation}
The leader
election algorithm is associated with incomplete binary trees, and is related
to various algorithms discussed in volume 3 of the Treatise \cite{knuth} of
\textsc{Knuth}.  As we will explain in section~\ref{sec:leaderelection}, a linear
combination of Bernoulli numbers quite similar to the one from Equation
\eqref{eq:em} already occurs there, as well as another alternating sum not involving
Bernoulli numbers, and both display the same type of asymptotics with a main
divergence and a sub-leading term which is (surprisingly, at first)
$1$-periodic in $\log_2(m)$ (possibly up to a multiplicative factor $m$).  Such
results and others served as a test bed for various techniques constitutive of
``analytic combinatorics'' \cite{flajosedgebook}.  Particularly, the
\emph{method of Mellin transform} was developed by \textsc{Flajolet} and
collaborators to handle a variety of such problems of asymptotics.  The
references \cite{flajgrabkirsprodtich1994} and \cite{flajgoli1994} give a nice
introduction.
See also the recent historical survey
contributed by \textsc{Prodinger} \cite{prodinger2022}.

Next, now with $b>1$ arbitrary, we explain that the quantities
$(u_m(s))_{m\geq0}$ from Equation \eqref{eq:umber} originate in the method of
moments we introduced into the topic of Dirichlet series with missing digits
(\cite{burnolkempner, burnolzeta}).  We prove the following representation of
$\zeta(s)$, away from the poles of the coefficients:
\begin{equation}\label{eq:zetaum}
  \zeta(s) = \sum_{1\leq n < b^{\ell-1}} n^{-s} + \sum_{m=0}^\infty (-1)^m
  u_m(s)\frac{(s)_m}{m!}\sum_{b^{\ell-1}\leq n < b^{\ell}} n^{-m-s}\,.
\end{equation}
Here, $\ell$ is some arbitrary integer at least equal to $2$.  We will
show that, locally uniformly and away from poles, $u_m(s)(s)_m/m! =
O(m^{-1})$. In other words, $m^su_m(s)=O(1)$. Thus, these series converge
geometrically fast.  See Theorem \ref{thm:O1/m} for the precise statement.

The finite parts at $s=1$ give series for the Euler-Mascheroni constant, which
we studied in \cite{burnoleuler} (for $b=2$).  Building upon the methods from
\cite{burnoleuler}, we obtain in Theorem \ref{thm:asympum} for any $s$ (not a
pole) the asymptotics of $m^s u_m(s)$ to all orders in inverse powers
$m^{-j}$, $j\geq0$. The numerators are $1$-periodic functions in $\log_b(m)$.

Despite rather extensive search and surveying of the existing literature, the
author has not identified any earlier occurrence of similar completely
explicit asymptotic expansions having the typical oscillations at each order. The key
technical tool, beyond the general idea of the Mellin transform representation
of power sums, is a statement about the behavior of the ratios of two Gamma functions
$\Gamma(z+\alpha)/\Gamma(z+\beta)$, which was established by the author in
\cite{burnoleuler}.  This ratio is, of course, not something new, and
\textsc{Flajolet} and \textsc{Odlyzko} in their foundational paper
\cite{flajoodlyz1990} on singularity analysis give its asymptotic expansion
for $\alpha$ and $\beta$ fixed. They were apparently unaware that this result
(in some equivalent form) is usually attributed to \textsc{Tricomi} and
\textsc{Erd\'elyi} \cite{tricomi1951} who published it in 1951.  The key fact
though is that we need some version allowing unbounded values of the
parameters $\alpha$ and $\beta$.  This is what the author established in
\cite{burnoleuler}.

The $(u_m(s))$ sequence has an interesting exponential generating function
(see \cite{burnolzeta}).  One of the general strategies in analytic
combinatorics first obtains the asymptotics of the generating function (to
some order), and then, typically via a saddle-point analysis, proceeds with
the ``de-poissonization'' (\cite[VIII.5.3]{flajosedgebook},
\cite[\S10]{szpbook2001}) to establish the asymptotics (to some order) of the
original coefficients.  It seems to us, that such a ``detour'' would make very
arduous the complete description of the asymptotic expansion to all orders, as
has been achieved here, via a direct study of the $u_m(s)$ quantity itself.

\section{Notation}

Throughout this paper, $b$ is an integer at least $2$ (and is set
equal to $2$ only in the next section).  The roots of the equation $b^z=1$ are
the $\chi_k = (\log b)^{-1}2\pi i k$ for $k\in\ZZ$.

The Pochhammer symbol is $(z)_n=
\prod_{0\leq j<n} (z+j)=\Gamma(z+n)/\Gamma(z)$.  In particular $(z)_0=1$. The
binomial symbol $\binom{z}{q}$ will be used, for a complex number $z$, to
mean $z(z-1)\dots(z-q+1)/q!$, equivalently
$(z-q+1)_q/q!$ or $\Gamma(z+1)/\Gamma(z-q+1)/q!$.

Bernoulli numbers are denoted $B_n$ (with the convention $B_1 = -\frac12$) and
the Bernoulli polynomials use the same letter, so $B_n = B_n(0)$. The
use of the same letter for both should not cause confusion.

\section{Leader election algorithm}\label{sec:leaderelection}

Let $n$ be a positive integer.  Consider a group of $n$ people throwing
independent unbiased coins.  On the first round, those who throw tails are
eliminated.  If everybody has thrown tails, (or heads), nobody gets
eliminated.  Those non-eliminated continue the game until only one remains.
With probability one, this completes after finitely many rounds.  Let $H_n$ be
the random variable equal to the number of rounds needed to reach that final
stage (so $H_n=\infty$ with probability zero).  Let $G_n(z)$ be the (formal
power series, at first) generating function of its distribution:
$G_n(z) = \sum_{k=0}^\infty \PP(H_n=k)z^k$.  In particular $G_1(z) =
1$.  As for $G_0(z)$ it can be left undefined, but it is reasonable and useful
to declare it to be $0$ (it is not a probability generating function then).
Suppose $n\geq2$, then $H_n\geq1$ necessarily and for $k\geq0$:
\begin{equation*}
  \PP(H_n = k+1) = 2\cdot 2^{-n}\PP(H_n=k) + \sum_{j=1}^{n-1}\binom{n}{j}2^{-n}\PP(H_j=k).
\end{equation*}
This follows simply by letting $j$ be the number of winners on the first
round.  The case $j=0$ is treated separately, as explained above, and
contributes in the same way as does $j=n$.  At the level of the generating
functions this gives the linear recurrence
\begin{equation}\label{eq:Gnrec}
  n\geq2 \implies G_n(z) = \frac{z}{2^n - 2z}\sum_{j=1}^{n-1}\binom{n}{j}G_j(z).
\end{equation}
Hence, $G_2(z) = \frac{z}{2-z}$, $G_3(z) = \frac{3z}{(4-z)(2-z)}$, and by
induction for $n\geq1$, $G_n(z)$ times $\prod_{1\leq j<n}(2^j -z)$ is a
polynomial of degree at most $n-1$ (as the limit of the rational function at
infinity is finite, by induction).  We knew already, from its definition as a
probability generating function, that the analytic function $G_n(z)$ was well
defined for $|z|\leq 1$, with value $1$ at $z=1$. We now know that the radii
of convergence are all at least $2$ (the next proposition shows they are exactly
equal to $2$ for $n\geq2$).  In particular, $H_n$ has all its moments finite.
We let $d_n = \EE(H_n)$.

We now solve the recurrence (this whole topic is well-known
\cite{prodinger1993, flajosedge1995, fillmahmszp1996}). We provide a detailed
proof here for the convenience of the reader.
\begin{prop}
  We have, for $n\geq1$: 
  \begin{equation}\label{eq:Gn}
    G_n(z) = \sum_{k=0}^{n-1}\binom{n}{k}\frac{2^k (1-z)B_k}{2^k -z}\;.
  \end{equation}
\end{prop}
\noindent Notice that for $n\geq2$, we must obtain $G_n(0) = 0$ and indeed it is true that
  $\sum_{k=0}^{n-1}\binom{n}{k}B_k = B_n(1)-B_n(0) = 0$ for $n\geq2$.
\begin{proof}
  We consider the exponential generating function of the $G_n(z)$'s, as a formal
  power series (with $G_0=0$, so the constant term in $w$ vanishes):
  \begin{equation*}
    G(z,w) = \sum_{n=0}^\infty G_n(z) \frac{w^n}{n!}\;.
  \end{equation*}
  From Equation \eqref{eq:Gnrec}, and with the usual notation for extracting 
  coefficients of formal power series:
  \begin{align*}
    n\geq2\implies 2^n G_n(z) &= z G_n(z) +  z \sum_{j=0}^{n}\binom{n}{j}G_j(z)
\\    &= z G_n(z) + z n! [w^n] (\e^w G(z,w))
\\
    n\geq2\implies [w^n] G(z, 2w) &= [w^n] (z(1 + \e^w) G(z,w)).
  \end{align*}
  So $G(z,2w) - z(1+\e^w)G(z,w)$ is reduced to its linear term in $w$ which is
  $2w - 2zw=2(1-z)w$ and we obtain
  \begin{equation*}
    G(z,2w) = 2(1-z)w + z(\e^w+1)G(z,w).
  \end{equation*}
  As $G(z,w)$ has no constant term in $w$, it makes sense to divide in the
  ring of formal power series in $w$ by $\e^{2w}-1 = (\e^w -1)(\e^w+1)$.  We obtain:
  \begin{align*}
    C(z,w) &= \frac{G(z,w)}{\e^w-1}\\
    C(z,2w) &= \frac{(1-z)2w}{\e^{2w}-1} + z C(z,w)\\
\implies [w^n]C(z,w) &= \frac{1-z}{2^n -z}[w^n]\frac{2w}{\e^{2w}-1} = \frac{1-z}{2^n -z}\frac{2^n B_{n}}{n!}\;.
  \end{align*}
  Hence:
  \begin{equation*}
    G(z,w)= (\e^w-1)\sum_{k=0}^\infty\frac{2^k (1-z)B_k}{2^k -z}\frac{w^k}{k!}\;.
  \end{equation*}
  Multiplying out and extracting $G_n(z) = n![w^n]G(z,w)$ we obtain Equation \eqref{eq:Gn}.
\end{proof}
Let $s$ be some complex number such that $2^{1-s}\not\in\{2^k, 0\leq k < n\}$.
Substituting $z = 2^{1-s}$ in Equation \eqref{eq:Gn} we obtain the formula:
\begin{equation*}
  G_n(2^{1-s}) = (1 - 2^{1-s})\sum_{k=0}^{n-1}\binom{n}{k}\frac{2^{k+s} B_k}{2^{k+s} -2}\;,
\end{equation*}
and comparison with Equation \eqref{eq:umber} gives, for $b=2$, an expression
of $u_m(s)$ in terms of the probability generating function $G_{m+1}(z)$:
\begin{equation*}
  (m+1)\frac{2^s -2}{2^s} u_{m}(s) = G_{m+1}(2^{1-s})\,.
\end{equation*}
 
Defining $e_m$ to be the opposite of the derivative with respect to $s$ at
$s=1$ of the left-hand side, divided by $\log(2)$, we obtain $e_m =
G_{m+1}'(1) = \EE(H_{m+1})$ and Equation \eqref{eq:em}.

\textsc{Prodinger} described the asymptotics of $d_n = \EE(H_{n})$ as $d_n
\sim_{n\to\infty} \log_2(n) + \frac12 - \delta(\log_2(n))$ where $\delta$ is
an explicit $1$-periodic function with zero average and small amplitude, but
what the symbol $\sim$ meant was not defined in that paper. In the survey
\cite{flajosedge1995}, \textsc{Flajolet} and \textsc{Sedgewick} wrote it
(essentially) as $d_n = \log_2(n) + \frac12 - \delta(\log_2(n)) + O(\sqrt n)$
(see \cite[Ex.\@ 4]{flajosedge1995}) which may have been a typographical
mistake, with $O(1/\sqrt n)$ being intended.  In the book of
\textsc{Szpankowski} \cite{szpbook2001}, the result with error $O(n^{-1})$ is
attributed to \textsc{Prodinger} \cite{prodinger1993} and given as Theorem
10.29, part of Exercise 10.15, whose solution presumably uses the
poissonization-depoissonization approach explained in that chapter (hence, a
very different approach from the one sketched in \cite{prodinger1993} and
\cite[p.\@ 113]{flajosedge1995}, whose proof remained incomplete).

In \cite{burnoleuler} we confirmed the $O(n^{-1})$ estimate by two distinct
methods, the second one actually using the Mellin transform formula from both
\cite{prodinger1993} and \cite[p.\@ 113]{flajosedge1995}, providing a full
asymptotic expansion in inverse powers of $n$.

Volume 3 of the Treatise of \textsc{Knuth} \cite{knuth} contains at least
two other significant examples of a nature related with Equation \eqref{eq:em}:
  \begin{itemize}
  \item In \cite[\S5.2.2, p.\@ 130-133]{knuth}, there is an analysis, with
    credits to \textsc{de~Bruijn}, of the behavior for $n\to\infty$ of
    \begin{equation*}
      U_n = \sum_{2\leq k\leq
        n}\binom{n}{k}\frac{(-1)^k}{2^{k-1}-1}\;.
    \end{equation*}
    It is proven there that $U_n =
    n\log_2(n) + cn + n\delta(n) + O(1)$ where $c$ is some constant and
    $\delta$ is a $1$-periodic function of $\log_2(n)$, which has zero
    average.  Further, the Exercise 5.2.2-54 in that reference proposes
    another approach, leading to an exact formula for $U_n$ (see also
    \cite[p.\@ 112]{flajosedge1995}).  But, to the best of the author's
    knowledge, the $O(1)$ never got rigorously justified in the literature on
    the basis of this second approach (the answer in \cite{knuth} only
    contains an elliptical indication).
  \item A second example, closer to Equation \eqref{eq:em}, is given by
    \begin{equation*}
      \sum_{2\leq k<n} \binom{n}{k}\frac{B_k}{2^{k-1}-1}\;.
    \end{equation*}
    (\cite[\S6.3, eq.\@ (18)]{knuth}).  It is accompanied with the comment
    ``\emph{This formula is probably the hardest asymptotic nut we have yet
      had to crack}''.  \textsc{Knuth} proves that this sequence has an
    asymptotic $n\log_2(n/\pi) + cn + n\delta(n) + O(1)$, where again $c$ is
    some constant and $\delta(n)$ is a $1$-periodic function of $\log_2(n)$
    with zero mean. This is established in the answer to exercise 6.3-34 of that reference
    (this answer on page 727 was updated on the occasion of the 52nd printing).
\end{itemize}

\section{Dirichlet series with missing digits}

Let us consider a ``restricted'' Dirichlet series
$K(s) = \sum'n^{-s}$, which keeps the positive integers $n$ whose radix-$b$
representation uses only certain allowed digits among $\Iffint{0}{b-1}$.
Soon, we will allow all digits, and only consider the Riemann zeta
function, but the general case brings perspective.

It has been known since the work of \textsc{Allouche}, \textsc{Mendès-France}, and
\textsc{Peyrière} \cite{allouchemendespeyriere2000}
(see also
\cite{allouchecohen1985}) that each such series admits a meromorphic
continuation to the whole complex plane.  This important result of
\cite{allouchemendespeyriere2000} is even more general as it applies to
\emph{automatic} Dirichlet series. See \cite{alloshalstip2025} for a recent
paper on this and related topics.  The analytic continuation is established
thanks to an ``infinite functional equation'' which expresses linearly $K(s)$
in terms of $K(s+1)$, $K(s+2)$, \dots.  This relation had appeared in
proto-form, for the case of harmonic series with one forbidden digit, in the
work of \textsc{Baillie} \cite{baillie1979}, who had obtained the various
$\sum' n^{-1}$ (i.e.\@ $s=1$) where one of the ten decimal digits is
forbidden, each to twenty decimal places.  This 1979 work of \textsc{Baillie}
is based upon the observation that one can compute the partial restricted
harmonic sums contributed by integers having $\ell+1$ decimal digits if one
knows the inverse power sums with exponents $1$, $2$, $3$, $4$, \dots, for
(restricted) integers with only $\ell$ digits.  In turn, these inverse power
sums can be expressed in terms of those with integers having only $\ell-1$
digits, etc\dots.  This is the basis of the \textsc{Baillie} algorithm.
Considering generally $n^{-s}$ and summing over all $\ell$'s one obtains an
infinite functional equation, which is a special case of those of
\cite{allouchemendespeyriere2000}.

Recently, the author has put forward another approach which uses certain
measures $\mu_s$ (for $s$ in the half-plane of convergence) on the half-open
interval $\Ifo{0}{1}$ and expresses $K(s)$ in terms of inverse powers of
integers having at most $\ell$ digits, using as coefficients the moments
of the measure $\mu_s$.

We focus in this paper on the case of the Riemann zeta function.  Let us give
the details of the construction then.

One fixes a \emph{level} $\ell$ which is an integer at
least $2$; the case with $\ell=1$ can also be considered, but requires
some additional discussion of convergence aspects (see \cite[Th.\@
1]{burnolzeta} for $\Re s>1$).  The starting point is:
\begin{equation}\label{eq:zetamu}
  \zeta(s) = \sum_{1\leq n < b^{\ell-1}} n^{-s} + \sum_{b^{\ell-1}\leq n <
    b^{\ell}} \int_{\Ifo{0}{1}}\frac{\dmu_s(x)}{(n + x)^s}\,
\end{equation}
where $\mu_s$ is the discrete complex measure (not depending on $\ell$) defined as
\begin{equation*}
  \mu_s = \sum_{j=0}^\infty b^{-js}\sum_{0\leq k < b^j}\delta_{b^{-j} k}\,.
\end{equation*}
As $\Re(s)=\sigma>1$, $\mu_s$ is a
complex measure.  Its total variation measure $|\mu_s|$ is $(1 -b^{-\sigma})|1 -
b^{-s}|^{-1}\mu_\sigma$, and has total mass
$(1 -b^{-\sigma})/(1 -b^{1-\sigma})/|1-b^{-s}|$.

The vocabulary of measures ceases to apply for $\Re s\leq 1$, but the moments
will allow the analytic continuation.
They are defined as:
\begin{equation}
  \label{eq:umdef}
  u_m(s) = \int_{\Ifo01}x^m \dmu_s(x)=
  0^m +\sum_{j=1}^\infty \bigl(\sum_{0\leq k < b^j} (k/b^j)^m\bigr)b^{-js}\,.
\end{equation}
Using these moments and the binomial series, we can rewrite Equation \eqref{eq:zetamu}
into the form of Equation \eqref{eq:zetaum}  which is the one of interest to us here:
\begin{equation*}
  \zeta(s) = \sum_{1\leq n < b^{\ell-1}} n^{-s} + \sum_{m=0}^\infty (-1)^m
  u_m(s)\frac{(s)_m}{m!}\sum_{b^{\ell-1}\leq n < b^{\ell}} n^{-m-s}\,.
\end{equation*}
One computes $u_0(s) = b^s/(b^s - b)$ and establishes the linear recurrence for $m\geq1$
(\cite[Prop.\@ 1]{burnolzeta}):
\begin{equation}
  \label{eq:umrec}
  m\geq 1\implies u_m(s) = \frac1{b^{m+s} - b} \sum_{j=1}^m \binom{m}{j}\gamma_j u_{m-j}(s)\,,
\end{equation}
where one has set
\begin{equation*}
  \gamma_j = \sum_{0\leq d< b} d^j\,.
\end{equation*}
The recurrence \eqref{eq:umrec} implies a meromorphic continuation of $u_m(s)$
to the whole complex plane.

Indeed, $u_m(s)$ is actually the quantity from
Equation \eqref{eq:umber}: this follows from \eqref{eq:umdef}
via the expression of power sums in terms of Bernoulli numbers and
polynomials:
\begin{align*}
  \sum_{0\leq n < N} n^m &= \frac{B_{m+1}(N)-B_{m+1}}{m+1}
\\
  B_{m+1}(N)
  &= \sum_{j=0}^{m+1}\binom{m+1}{j}B_{j}N^{m+1-j}.
\end{align*}
Using this for $N=1$, we obtain for $m\geq1$ that
$\sum_{j=0}^{m}\binom{m+1}{j}B_{j}=B_{m+1}(1) - B_{m+1}(0)=0$, so \eqref{eq:umber} can equivalently
be written, for $m\geq1$, as
\begin{equation}\label{eq:umber1}
 u_m(s) =     \frac1{m+1}\sum_{j=0}^{m}\binom{m+1}{j}B_j\frac{b}{b^{s+j}-b}.
\end{equation}
More explicitly, for $m\geq1$:
\begin{equation*}
u_m(s) =  
\frac1{m+1}\frac{b}{b^s -b} - \frac{b}{2(b^{s+1}-b)} +
\sum_{1\leq k \leq \lfloor \frac m2 \rfloor}\frac{m!}{(m-2k+1)!}\frac{B_{2k}}{(2k)!}\frac{b}{b^{s+2k}-b}.
\end{equation*}
This shows that simple poles are located at the roots of $b^s=b^a$,
for $a$ in $\{1,0,-1,-3,\dots,1-2\lfloor \frac m2\rfloor\}$.

The ``explicit formulas'' \eqref{eq:umber} and
\eqref{eq:umber1} are
unstable numerically, having large contributions of alternating signs (for $s$
real).  In contrast, the recurrence Equation \eqref{eq:umrec} is stable
numerically (more precisely, both numerical explorations and the method of proof of \cite[Th.\@ 2.3]{burnolcont} indicate that, once $\Re(s)+m>1$, computing $u_m(s)$ via the recurrence creates no numerical perils from adding quantities of opposite signs). Another method to obtain Equations \eqref{eq:umber} or \eqref{eq:umber1}
is to first compute the exponential generating function of the moments, see
\cite[\S4]{burnolzeta}.

We show that Equation \eqref{eq:zetaum} is valid throughout the complex plane,
as a corollary to \cite[Th.\@ 2.3]{burnolcont}.
\begin{theo}\label{thm:O1/m}
  Let $\alpha_b(s)$ be the entire function
  \begin{equation*}
    (1
    -b^{1-s})(1-b^{-s})\prod_{p=1}^\infty (1 - b^{1 - 2p - s}).
  \end{equation*}
  Uniformly
  with respect to $s$ in any compact subset of the complex plane, there
  holds \[\alpha_b(s)\frac{(s)_m}{m!} u_m(s) \mathop{=}\limits_{m\to\infty}
    O(m^{-1})\,,\] and Equation \eqref{eq:zetaum}, multiplied by
  $\alpha_b(s)$, thus represents $\alpha_b(s)\zeta(s)$ throughout the complex
  plane, via a series (for each given $\ell\geq2$) of entire functions, which
  converges geometrically fast, locally uniformly.
\end{theo}
\begin{proof}
  Let $\alpha(s)=\prod_{n=0}^\infty (1 - b^{1 -n -s})$. In \cite[Th.\@
  2.3]{burnolcont} (whose parameter $\lambda$ is $1$ for our case) it is
  established, for some more general coefficients $u_m(s)$ related to zeta
  series $\sum' n^{-s}$, where the sum keeps only those positive integers $n$
  whose radix-$b$ representation uses only some specific digits, that
  $\alpha(s)\frac{(s)_m}{m!} u_m(s)=O(m^{-1})$ (where $\alpha(s)u_m(s)$ is an
  entire function) holds uniformly on any compact subset
  of the complex plane.  Let us consider a rectangular contour not going
  through any of the zeros of $\alpha(s)$.  Replacing $\alpha(s)$ by
  $\alpha_b(s)$ on the contour is like multiplying by $\alpha_b(s)/\alpha(s)$ which has
  no singularity there, and the
  functions $\alpha_b(s)\frac{(s)_m}{m!} u_m(s)$, $m\geq1$, are thus on this
  contour $O(m^{-1})$,
  uniformly with respect to both $s$ and $m\geq1$.  According
  to Equation \eqref{eq:umber} all functions $\alpha_b(s)u_m(s)$, for
  $m\geq1$, are entire (this is also true for $m=0$ as $u_0(s)=b^s/(b^s-b)$).
  So the maximum modulus principle gives us the uniform $O(m^{-1})$ bound on
  the filled rectangle.

  Equation \eqref{eq:zetaum} gives an expression for $\alpha_b(s)\zeta(s)$ for
  $\Re s>1$ as a series of entire functions, which
  converges everywhere in the complex plane, and locally uniformly.  Hence,
  the identity thus obtained is valid throughout $\CC$.
\end{proof}
It proves
convenient to also consider the quantities
\begin{equation*}
  u_m^*(s) = u_m(s)\frac{(s+1)_m}{m!},
\end{equation*}
which verify the next recurrence (\cite[Eq.\@ (24)]{burnolzeta}):
\begin{equation}\label{eq:umstarrec}
  m\geq1\implies u_{m}^*(s) = \frac1{b^{m+s} -b}
  \sum_{j=1}^{m} \binom{s+m}{j}\gamma_ju_{m-j}^*(s).
\end{equation}
With these new coefficients, Equation \eqref{eq:zetaum} reads:
\begin{equation*}
  \zeta(s) = \sum_{1\leq n < b^{\ell-1}} n^{-s} + \sum_{m=0}^\infty 
  \frac{(-1)^ms\,u_m^*(s)}{m+s}\sum_{b^{\ell-1}\leq n < b^{\ell}} n^{-m-s}\,.
\end{equation*}
For a fixed $s$ distinct from the zeros of $\alpha_b(s)$, Theorem \ref{thm:O1/m} implies in
particular that $u_m^*(s)$ is bounded as $m\to\infty$, as $u_m^*(s) =
(s+m)s^{-1} \frac{(s)_m}{m!} u_m(s)$.

\section{Contour integrals and residues}

Let $m\geq1$ and consider the continuous function on the real line,
exponentially small at infinity, which is defined as $\phi_m(u) = (1 -
\e^u)^m\e^{u/2}$ for $u\leq 0$ and $0$ for $u\geq0$. The Fourier transform
$\widehat{\phi_m}(\xi) = \int_{\RR}\e^{i\xi u}\phi_m(u)\du$ is, with
$s=\frac12+i\xi$ and a change of variable:
\begin{equation*}
 \int_{-\infty}^0(1 - \e^u)^m\e^{u/2}\e^{i\xi u}\du 
= \int_0^1(1-t)^mt^{s-1}\dt = \sum_{j=0}^m(-1)^j\frac{\binom{m}{j}}{s+j}\,.
\end{equation*}
This last result is $\frac{m!}{s(s+1)\dots(s+m)}$, which is absolutely
integrable on the line $\Re s = \frac12$.  The Fourier inversion formula
gives (with $\dz=i\mathrm{d}\xi$ for $z=\frac12+i\xi$):
\begin{equation*}
\forall u\in\RR\quad  \phi_m(u) = \frac1{2\pi i}\int_{\frac12-i\infty}^{\frac12+i\infty}
              \frac{m!}{z(z+1)\dots(z+m)}\e^{-(z-\frac12)u}\dz\,.
\end{equation*}
This can be re-written into this \emph{Mellin inversion formula}:
\begin{equation}\label{eq:mellin}
x>0\implies \frac1{2\pi i}\int_{\frac12-i\infty}^{\frac12+i\infty}
              \frac{m!}{z(z+1)\dots(z+m)}x^{-z}\dz =
              \begin{cases}
                (1-x)^m& (x\leq 1)\,,\\
                    0 &  (x\geq 1)\,.
              \end{cases}
\end{equation}
We do not need to consider here the more precise
\emph{Perron formulas} which estimate the error if integrating only from
$\frac12-iT$ to $\frac12+iT$ (cf.\@ \cite[3.19]{titchmarsh},
\cite[II.2]{tenenbaumGSM}).  And we don't need the $m=0$ case with its jump at
$x=1$.

The line of integration in Equation \eqref{eq:mellin} can be shifted to any
given positive real part $\Re z =a>0$.  Let $s$ be given with real part
greater than $1$, and let us compute $u_m(s)$, for $m\geq1$, from Equation
\eqref{eq:umdef} using Equation \eqref{eq:mellin}:
\begin{align}
\notag
  u_m(s) &= \sum_{j=1}^\infty \bigl(\sum_{0\leq k < b^j} (k/b^j)^m\bigr)b^{-js}
\\
\notag
&= \sum_{j=1}^\infty \bigl(\sum_{0< k \leq  b^j} (1 - k/b^j)^m\bigr)b^{-js}
\\
&= \sum_{j=1}^\infty \sum_{k=1}^\infty \frac1{2\pi i}\int_{a-i\infty}^{a+i\infty}
              \frac{m!}{z(z+1)\dots(z+m)}k^{-z}b^{jz}\dz\; b^{-js} 
\end{align}
Choosing in the above equation $a>1$, we can permute the inner sum with the integration:
\begin{equation*}
u_m(s)  = \sum_{j=1}^\infty \frac1{2\pi i}\int_{a-i\infty}^{a+i\infty}
  \frac{m!}{z(z+1)\dots(z+m)}\zeta(z)b^{j(z-s)}\dz\,,
\end{equation*}
and if we now also impose $a<\Re s$, we can permute the remaining sum with the integration:
\begin{equation}\label{eq:ummellin}
\Re s >a > 1 \implies  u_m(s)=
  \frac1{2\pi i}\int_{a-i\infty}^{a+i\infty} \frac{m!}{z(z+1)\dots(z+m)}\zeta(z)
   \frac{b^{z}}{b^s - b^{z}}\dz\,.
\end{equation}
The formula is also valid with $b^s$ replacing $b^z$ in the numerator, as
$\zeta(z)$ is bounded for $\Re z\geq a>1$ and we can shift the abscissa of
integration to $+\infty$, showing the difference to be zero; or we start the
computation with the summation in Equation \eqref{eq:umdef} starting at $j=0$,
not $j=1$.

If in Equation \eqref{eq:ummellin} we shift the abscissa of integration to the
left, picking up the residues at the poles, we reconstitute Equation
\eqref{eq:umber} with its Bernoulli numbers, but this requires discussing a
subtle point: once we have picked the residue at $z= 1-m$ (if $m$ is even),
the line integral
slightly to its left actually vanishes (fortunately, as it ceases to be
absolutely integrable when reaching $\Re z = -m+\frac12$ a bit farther to the
left).  We do not need (contrary to a similar situation encountered in
\cite[p.\@ 113]{flajosedge1995}) to search for some extra explanation for this
fact here, as we proved \eqref{eq:ummellin} directly, and the vanishing of the
contour integral expresses its compatibility with \eqref{eq:umber}.

We are much more interested in shifting the abscissa of integration to $\Re
z > \Re s$, picking up the (opposite of the) residues at the simple poles
where $b^z= b^s$.  Once this is done (using intermediate rectangular
contours going mid-way between successive poles of $1/(b^s-b^z)$), the line integral on
$\Re z = \Re s + \eta$, $\eta>0$, vanishes because both $\zeta(z)$ and
$b^z/(b^s-b^z)$ are bounded for $\Re z \geq \Re s + \eta$.  So $u_m(s)$ is
\emph{exactly} the (opposite of the) sum of the residues over the simple poles
at $z = s + \chi_k$, $k\in\ZZ$, $\chi_k = (\log b)^{-1}2\pi i k$.  The
obtained expression is suitable for some analytic continuation, as is
explained in the next Theorem.
\begin{theo}\label{thm:um}
  The meromorphic function $u_m(s)$, $m\geq1$, from Equation \eqref{eq:umber},
  is represented in the open half-plane $\Re s > -m+\frac12$, away from its
  poles at $s = q-\chi_k$, $\chi_k = (\log b)^{-1}2\pi i k$, $k\in\ZZ$, $q=1$,
  $0$, $-1$, $-3$, \dots{}, $\geq -(m-1)$, by the absolutely convergent series:
  \begin{align}\label{eq:umres}
    u_m(s) &= \frac1{\log b}\sum_{k\in\ZZ}
    \frac{m!\zeta(s+\chi_k)}{(s+\chi_k)\cdots(s+\chi_k+m)}
    \\\notag
    &=\frac1{\log b}\sum_{k\in\ZZ}
    \frac{\Gamma(m+1)}{\Gamma(s+\chi_k+m+1)}\Gamma(s+\chi_k)\zeta(s+\chi_k)
  \end{align}
\end{theo}
\begin{proof}
  Equation \eqref{eq:umres} is obtained for $\Re s >1$ as explained previously
  via the calculus of residues applied to Equation \eqref{eq:ummellin}.  Due
  to $|\zeta(z)|=O(|\Im z|^{A})$ for $|\Im z|\geq 1$, $\Re z \geq -A+\frac12$,
  $A>\frac12$ (\cite[Eq.\@ (5.1.1)]{titchmarsh}), the series in Equation
  \eqref{eq:umres} converges absolutely for $\Re s > -m+\frac12$, for $s$ not
  a pole, and it does so uniformly on compact subsets of this open half-plane, if
  not containing any pole. It thus gives the analytic continuation of $u_m(s)$
  from $\Re s>1$, where the identity has been proven, to $\Re s > -m +\frac12$
  (punctured at the poles). The trivial zeros of $\zeta(s)$ at $-2$, $-4$,
  \dots, remove candidate poles at $s=-2p-\chi_k$, $p\geq1$, $2p\leq m-1$, in
  accordance with Equation \eqref{eq:umber}.  The second line is a direct
  reformulation.
\end{proof}
It looks as if Equation \eqref{eq:umres} says that $u_m(s)$ has poles at
$s=-m-\chi_k$, if $m$ is odd, but this is forgetting that its validity has
been established only for $\Re s > -m+\frac12$.  And indeed no such poles
exist in view of Equation \eqref{eq:umber}.

Switching to the sequence $(u_m^*(s))$ we obtain:
\begin{prop}\label{prop:umstar}
For $m\geq1$ and $s$ with $\Re s > -m + \frac12$, not a pole of the left-hand side:
  \begin{equation*}
    u_m^*(s) = 
    \frac1{\log b}\frac{\zeta(s)}{s}
    +
    \frac1{\log b}\sum_{\substack{k\in\ZZ\\k\neq0}}
    \frac{\Gamma(m+s+1)}{\Gamma(m+s+\chi_k+1)}
    \frac{\Gamma(s+\chi_k)\zeta(s+\chi_k)}{\Gamma(s+1)}
  \end{equation*}
\end{prop}
\begin{proof}
  As $u_m^*(s) = \frac{(s+1)_m}{m!}u_m(s)$, Equation \eqref{eq:umres} for $\Re
  s> -m+\frac12$ gives, at first for $s \neq q-\chi_k$, $k\in\ZZ$, $q=1$, $0$, $-1$,
  $-3$, \dots{}, $\geq -(m-1)$:
  \begin{equation*}
    \begin{split}
      u_m^*(s) = \frac1{\log b}\sum_{k\in\ZZ}
      \frac{(s+1)_m}{(s+\chi_k+1)_{m}}\frac{\zeta(s+\chi_k)}{s+\chi_k} 
\\=
      \frac1{\log b}\frac{\zeta(s)}{s} +
      \frac1{\log b}\sum_{\substack{k\in\ZZ\\k\neq0}}
      \frac{(s+1)_m}{(s+\chi_k)_{m+1}}\zeta(s+\chi_k).
    \end{split}
  \end{equation*}
  The multiplication by $(s+1)_m=(s+1)\dots(s+m)$ has removed the simple poles
  which were contributed by $k=0$ at $-1$, $-3$, \dots $\geq -(m-1)$, and has
  created ``trivial zeros'' at $-2$, $-4$, \dots, $\geq -(m-1)$.  Poles at
  $1-2p-\chi_k$, $k\neq 0$, $p\geq1$, $1-2p\geq -(m-1)$, remain.
\end{proof}

\section{Full asymptotics of the moments}

For the convenience of the reader we state here in full a result of the author on
ratios of two Gamma functions \cite[Th.\@
4]{burnoleuler}.  First, let us observe that there are unique polynomials
$Q_j(u,v)$ verifying the conditions $Q_0=1$,
  $Q_j(u,u)=0$ for $j>0$ and either one of the following two sets of
  recurrences:
  \begin{gather*}
\forall j>0, \quad Q_j(u+1,v)-Q_j(u,v) = u Q_{j-1}(u,v)
\\
\forall j>0, \quad Q_j(u,v+1)-Q_j(u,v) = -v Q_{j-1}(u,v+1).
  \end{gather*}
$Q_j$ is of total degree $2j$.
\begin{theo}[{\cite[Th.\@ 4]{burnoleuler}}]\label{thm:gammaratio}
  Let $c_1>0$.  Let $0<\eta<\pi$.  There exists $A>0$, such that for every
  positive integer $J$ there exists a constant $C_J$ such that for every
  complex number $z$ verifying $|z|\geq A$ and $|\Arg z|\leq\pi - \eta$, and
  for every pair of complex numbers $(\alpha, \beta)$ verifying the condition
  \begin{equation*}
    \max(|\alpha|^2,|\beta|^2)\leq c_1 |z|,
  \end{equation*}
  the following inequality holds:
  \begin{equation*}
    \left|\frac{\Gamma(z+\alpha)z^{\beta-\alpha}}{\Gamma(z+\beta)}
      - \sum_{0\leq j<J}\frac{Q_j(\alpha,\beta)}{z^j}\right|
    \leq C_J\frac{|\alpha-\beta|\max(1, |\alpha|, |\beta|)^{2J-1}}{|z|^J}\,,
  \end{equation*}
  with the polynomials $Q_j$ mentioned earlier.
\end{theo}
The above theorem is a quantitative (but non explicit) reinforcement of a well-known result of
\textsc{Tricomi} and \textsc{Erd\'elyi} \cite{tricomi1951}.

We now adapt the technique of \cite[Proof of Th.\@ 3]{burnoleuler},
and establish the full asymptotics of the coefficients $u_m^*(s)$, for a
given $s$.
\begin{theo}\label{thm:asymp}
  Let $s\notin(\{1,0\}+(\log b)^{-1}2\pi i \ZZ)\cup\bigcup_{p\geq 1}\bigl(
  1-2p+(\log b)^{-1}2\pi i (\ZZ\setminus\{0\})\bigr)$. Let $P_{0,s+1}=1$ and
  $P_{j,s+1}\in\CC[t]$ be defined by recurrence for $j\geq1$ by the conditions
  \begin{equation*}
P_{j,s+1}(0)=0, \qquad
    P_{j,s+1}(t+1)-P_{j,s+1}(t) = -(t+s+1)P_{j-1,s+1}(t+1)\,.
  \end{equation*}
   Let $\psi_{j,s}$ for
  $j\geq0$ be the $1$-periodic smooth function with zero mean
  \begin{equation*}
    \psi_{j,s}(t) =
    \sum_{\substack{k\in \ZZ\\k\neq 0}}
    P_{j,s+1}(\chi_k)\Gamma(s+\chi_k)\zeta(s+\chi_k)\e^{-2\pi i kt}\,.
  \end{equation*}
  Let $J$ be any positive integer.  There holds:
  \begin{equation}
    \label{eq:asymp}
    u_m^*(s) =   \frac1{\log b}\frac{\zeta(s)}{s}
    + \frac1{(\log b)\Gamma(s+1)}\sum_{0\leq j < J}
    \frac{\psi_{j,s}(\log_b(m))}{m^j} + O_{m\to\infty}(m^{-J})\,.
  \end{equation}
\end{theo}
\begin{proof}
  We note first that $\zeta$ having at most polynomial growth on each vertical
  line, and $\Gamma$ having exponential decrease, the series defining
  $\psi_{j,s}(t)$ for $t\in \RR$ is absolutely convergent, except perhaps if there
  is (at most one) $k\neq 0$ such that $s+\chi_k$ is a pole of $\zeta$ or of
  Gamma.  Considering the product $\Gamma(w)\zeta(w)$ to be already formed, it
  has poles at $1$, $0$, and at the odd negative integers.  So $s$ can
  cause a singular term to appear in the definition of $\psi_{j,s}$ only if it is one
  of $q - \chi_k$, for $k\in\ZZ$, $k\neq 0$, and $q\in\{1,0,-1,-3,\dots\}$.
  And the Theorem statement has excluded these values from consideration.

  For $s$ an even negative integer, $s=-2p$, the factor $\Gamma(s+1)^{-1}$ in
  Equation \eqref{eq:asymp} vanishes, as does $\zeta(s)$. So the estimate
  will certainly be valid if $u_m^*(-2p)=0$ for $m$
  large enough.  In fact $u_m^*(s)=(s+1)_m (m!)^{-1}u_m(s)$ and from
  \eqref{eq:umber}, $u_m$ is regular at every $s=-2p$, $p\geq1$, so
  $u_m^*(-2p)=0$ for $0<2p\leq m$.

  For $s$ an odd negative integer $1-2p$, $p\geq1$, again
  $\Gamma(s+1)^{-1}$ vanishes, so Equation \eqref{eq:asymp} is certainly valid
  if $u_m^*(1-2p) = (\log b)^{-1}\frac{\zeta(1-2p)}{1-2p}$
  for $m$ large enough.  We establish this indeed for $m\geq 2p$.
  From its definition, we obtain then
  \begin{equation*}
    u_m^*(1-2p) = \frac{(2p-2)!(m+1-2p)!}{m!}\Res_{s=1-2p} u_m(s)\,,
  \end{equation*}
   and from Equation \eqref{eq:umber}, using again $2\leq 2p\leq m$:
  \begin{align*}
    \Res_{s=1-2p}u_m(s) &= \frac{m!}{(m-2p+1)!(2p)!}(\log b)^{-1}B_{2p}
\\
  &= -(\log b)^{-1}\frac{m!}{(m+1-2p)!(2p-1)!}\zeta(1-2p)\,,
  \end{align*}
so, indeed, in conformity with Equation \eqref{eq:asymp},
$u_m^*(1-2p)  = -(\log b)^{-1}\frac{\zeta(1-2p)}{2p-1}$.

We now establish the validity of the asymptotic expansion Equation
\eqref{eq:asymp} for $s$ as in the Theorem statement, and not a negative
integer, so $\Gamma(s+1)$ has no pole.  From  Proposition \ref{prop:umstar}
we have, for $m>\max(0, \frac12 - \Re s)$,
  \begin{equation}\label{eq:umstarres}
    u_m^*(s) = \frac1{\log b}\frac{\zeta(s)}s 
+ \frac1{\log b}\sum_{\substack{k\in\ZZ\\ k\neq 0}}
    \frac{(s+1)\dots(s+m)}{(s+\chi_k)\dots(s+\chi_k+m)}\zeta(s+\chi_k)\,.
  \end{equation}
  For every natural integer $n$ (even for every real number) there holds
  $|s+n|\leq |s+\chi_k+n|$ if $k$ is such that $|\Im s + 2\pi (\log b)^{-1}
  k|\geq |\Im s|$.  This will be the case if $|\chi_k|\geq 2|\Im s|$.  So if
  $|\chi_k|>\sqrt{m}$ and $\sqrt{m}\geq 2|\Im s|$, there holds
  $\frac{|s+n|}{|s+\chi_k+n|}\leq 1$ for every natural integer $n$.

  Let $J$ be any positive integer.  Let $A$ be a non-negative integer such that
  $|\zeta(s+\chi_k)|=O_{k\to\pm\infty}(|k|^A)$ ($s$ is fixed in this proof). For 
  $m> \max(2J +A, 4|\Im s|^2,\frac12-\Re s)$,
  and $k$ verifying $|\chi_k|>\sqrt{m}$:
  \begin{align*}
    \left|\frac{(s+1)\dots(s+m)}{(s+\chi_k)\dots(s+\chi_k+m)}\zeta(s+\chi_k)\right|
&=O(\frac {|s+1|\cdots|s+2J+A|}{|s+\chi_k|\cdots|s+\chi_k+2J+A|} |k|^A)
\\&=O_{|k|\to\infty}(|k|^{-2J-1})\,.
  \end{align*}
  Hence, the summation over all $k$ such that $|\chi_k|>\sqrt m$ is
  $O(m^{-J})$.  Consequently, keeping in Equation
  \eqref{eq:umstarres} only the terms with $|\chi_k|\leq
  \sqrt{m}$ gives an approximation which differs from $u_m^*(s)$ by $O(m^{-J})$.

We express the general term for $k$ with $|\chi_k|\leq \sqrt{m}$ as:
\begin{equation*}
  \frac1{\log b}
    \frac{\Gamma(m+s+1)m^{\chi_k}}{\Gamma(m+s+1+\chi_k)}
    \Gamma(s+1)^{-1}\Gamma(s+\chi_k)\zeta(s+\chi_k)m^{-\chi_k}\,.
\end{equation*}
and observe that $m^{-\chi_k} = \exp(-2\pi i k \log_b(m))$.

There is some positive constant $c_1$ such that (recalling $m>0$, even $m\geq3$):
\begin{equation*}
  |\chi_k|\leq \sqrt m\implies \max(|s+1|^2, |s+1+\chi_k|^2)\leq c_1 m\,.
\end{equation*}

Using $z=m$, $\alpha=s+1$,
$\beta=s+1+\chi_k$, Theorem \ref{thm:gammaratio} allows us to replace the ratio
$\frac{\Gamma(m+s+1)m^{\chi_k}}{\Gamma(m+s+1+\chi_k)}$ by the finite
approximation $\sum_{0\leq j < J} P_{j,s+1}(\chi_k)m^{-j} +
O(|\chi_k|^{2J}m^{-J})$.  We refer the reader to \cite[Rem.\@ 8]{burnoleuler}
regarding the sequence of polynomials $(P_{j,s+1})_{j\geq0}$ ($a=s+1$ in the
notation of that Remark). There holds $P_{j,s+1}(t) = (-1)^j
\frac{(t)_j}{j!}B_j^{(1-t)}(s+1)$, where $B_j^{(\rho)}$ is a \emph{generalized
  Bernoulli polynomial} \cite{norlund1961, tricomi1951}.  As
$\Gamma(s+\chi_k)\zeta(s+\chi_k)m^{-\chi_k}$ has exponential decrease for
$|k|\to\infty$, this is still the case after multiplying it by
$|\chi_k|^{2J}$, and the combined error term over all $k$'s with $|\chi_k|\leq
\sqrt{m}$ is thus $O(m^{-J})$.  We then lift the restriction $|\chi_k|\leq
\sqrt{m}$, reconstituting the complete Fourier series (over $k\neq0$)
$\psi_{j,s}(\log_b(m))$, $0\leq j<J$. For more details of this reasoning, we
refer the reader to \cite[Proof of Th.\@ 3]{burnoleuler}.
\end{proof}

A repetition of the arguments, but starting from Equation \eqref{eq:umres},
leads to the similar result regarding the $(u_m(s))$ sequence.
\begin{theo}\label{thm:asympum}
  Let $s\notin\{1,0,-1,-3,\dots\}+(\log b)^{-1}2\pi i \ZZ$. Let $P_{0,1}=1$ and
  $P_{j,1}\in\CC[t]$ be defined by recurrence for $j\geq1$ by the conditions
  \begin{equation*}
P_{j,1}(0)=0, \qquad
    P_{j,1}(t+1)-P_{j,1}(t) = -(t+1)P_{j-1,1}(t+1)\,,
  \end{equation*}
so that $P_{1,1}(t) = -
  t(t+1)/2$, $P_{2,1}(t) = t(t+1)(t+2)(3t+1)/24$, and $P_{3,1}(t)=-t^2(t+1)^2(t+2)(t+3)/48$.
   Let $\phi_{j,s}$ for
  $j\geq0$ be the $1$-periodic smooth function
  \begin{equation*}
    \phi_{j,s}(t) =
    \sum_{k\in \ZZ}
    P_{j,1}(s+\chi_k)\Gamma(s+\chi_k)\zeta(s+\chi_k)\e^{-2\pi i kt}\,.
  \end{equation*}
  Let $J$ be any positive integer.  There holds:
  \begin{equation*}
    u_m(s) = 
   \frac{m^{-s}}{\log b}\sum_{0\leq j < J}
    \frac{\phi_{j,s}(\log_b(m))}{m^j} +
    O_{m\to\infty}(m^{-\Re s -J})\,.
  \end{equation*}
\end{theo}
\begin{proof}
  We start from Equation \eqref{eq:umres}, and then use
  Theorem \ref{thm:gammaratio}
   with $z=m$, $\alpha=1$, $\beta=\alpha+s+\chi_k$.  Details are exactly as
  in the proof of Theorem \ref{thm:asymp} (with fewer complications
  in the treatment of negative integer values of $s$), and are left to the
  reader.
  See also \cite[Rem.\@ 3 \& 8]{burnoleuler} about the $P_{j,1}$ polynomials.
\end{proof}
\begin{rema}
  In view of Equation \eqref{eq:umber}, it may be legitimate to think that the
  indexing should be by $m+1$, not by $m$.  Both Theorems \ref{thm:asymp} and
  \ref{thm:asympum} can be slightly modified to provide expansions in inverse
  powers of $m+1$ and slightly different periodic functions in $\log_b(m+1)$.
  The proofs are almost identical, and the details are left to the reader.
\end{rema}

For numerical examples, see \cite{burnolosc}.  In that reference the general case
of Dirichlet series with missing digits is treated, but the oscillations of
the associated moments are then established only for $s$ in the half-plane of
absolute convergence of the series.


\clearpage

\singlespacing

\bigskip
\noindent
  Université de Lille,
  Faculté des Sciences et technologies,
  Département de mathématiques,
  Cité Scientifique,
  F-59655 Villeneuve d'Ascq cedex,
  France.
\newline
\strut \texttt{jean-francois.burnol@univ-lille.fr}

\end{document}